\documentclass[11pt]{article}
 \usepackage{amsmath}
\usepackage{amssymb}
\usepackage{verbatim}
  \usepackage{artmods,macros}
  \usepackage{amsmath}
  \usepackage{amssymb}
\newcommand{\khat}{\skew1\widehat{k}}

\makeatletter
\newcommand{\pmat}[1]{\begin{pmatrix}#1\end{pmatrix}} 

\ifx   \innerprod\undefined%
   \def\innerprod(#1,#2){\langle#1,#2\rangle} 
\fi

\newcommand{\Newcommand}[2]%
   {\ifx#1\undefined \newcommand{#1}{#2} \else \renewcommand{#1}{#2} \fi}
\ifx\mod\undefined
  \newcommand{\mod}[1]{|#1|}
\else
  \renewcommand{\mod}[1]{|#1|}
\fi
\ifx\Range\undefined
  \newcommand{\Range}{\mathop{\operator@font{range}}}
\else
  \renewcommand{\Range}{\mathop{\operator@font{range}}}
\fi
\Newcommand{\Re} {\mathbb{R}}           
\Newcommand{\subject}{\mathop{\operator@font{subject\ to}}}  

\providecommand{\det}  {\mathop{\operator@font{det}}}
\providecommand{\diag} {\mathop{\operator@font{diag}}}
\providecommand{\dim}  {\mathop{\operator@font{dim}}}
\providecommand{\exp}  {\mathop{\operator@font{exp}}}

\providecommand{\rank} {\mathop{\operator@font{rank}}}
\providecommand{\vec}  {\mathop{\operator@font{vec}}}

\newcommand{\Angle}{\mathop{\operator@font{angle}}}
\newcommand{\argmin}{\mathop{\operator@font{argmin}}}
\newcommand{\argmax}{\mathop{\operator@font{argmax}}}
\newcommand{\conv}{\mathop{\operator@font{conv}}}
\newcommand{\eig}{\mathop{\operator@font{eig}}}
\newcommand{\svSet}{\mathop{\operator@font{sv}}}
\newcommand{\algmult}{\mathop{\operator@font{am}}}
\newcommand{\geomult}{\mathop{\operator@font{gm}}}

\newcommand{\T}{^T\!}

\newcommand{\realpart}{\mathop{\operator@font{Re}}\,}

\newcommand{\mgap}{\;\;}

\newcommand{\Cpp}{C\raise3pt\hbox{\tiny++}}
\newcommand{\blackslug}{\hbox{\hskip 1pt \vrule width 4pt height 6pt depth 1.5pt
  \hskip 1pt}}

\newcommand{\maxim}{\mathop{\operator@font{maximize}}}
\newcommand{\minim}{\mathop{\operator@font{minimize}}}
\newcommand{\cond}{{\operator@font{cond}}}

\newcommand{\sep}{\mathop{\operator@font{sep}}}
\newcommand{\In}{\mathop{\operator@font{In}}}
\newcommand{\Signature}{\mathop{\operator@font{Sn}}}
\newcommand{\boundary}{\mathop{\operator@font{bnd}}}
\newcommand{\interior}{\mathop{\operator@font{int}}}
\newcommand{\Int}{\mathop{\operator@font{int}}}
\newcommand{\Null}{\mathop{\operator@font{null}}}
\newcommand{\columnRank}{\mathop{\operator@font{column\ rank}}}
\newcommand{\rowRank}{\mathop{\operator@font{row\ rank}}}
\newcommand{\Rank}{\mathop{\operator@font{rank}}}
\newcommand{\op}{\mathop{\operator@font{op}}}
\newcommand{\re}{\mathop{\operator@font{re}}}
\newcommand{\range}{\mathop{\operator@font{range}}}
\newcommand{\sign}{\mathop{\operator@font{sign}}}
\newcommand{\sgn}{\mathop{\operator@font{sgn}}}
\newcommand{\Span}{\mathop{\operator@font{span}}}
\newcommand{\trace}{\mathop{\operator@font{trace}}}

\newcommand{\Grad}{\nabla\!}

\newcommand{\cents}{\hbox{\operator@font{\rlap/c}}}
\newcommand{\half}  {{\textstyle\frac12}}

\newcommand{\norm}[1]{\|#1\|}

\newcommand{\spose}[1]{\hbox to 0pt{#1\hss}}

\newcommand{\wordss}[1]{\quad\text{#1}\quad}

\newcommand{\submax}{_{\mathrm{max}}}
\newcommand{\submin}{_{\mathrm{min}}}

\newcommand{\E}{_{\scriptscriptstyle E}}

\Newcommand{\R}{_{\scriptscriptstyle R}}

\newcommand{\starsymbol}{\ast}
\newcommand{\superstar}{^\starsymbol}

\newcommand{\SUPERSTAR}{^{\raise 0.5pt\hbox{$\starsymbol$}}}
\renewcommand{\SUPERSTAR}{\superstar}

\newcommand{\vstar}{v\superstar}
\newcommand{\wstar}{w\superstar}
\newcommand{\xstar}{x\superstar}
\newcommand{\ystar}{y\superstar}

\newcommand{\alphabar}{\skew3\bar\alpha}

\newcommand{\muhat}{\skew3\widehat \mu}

\providecommand{\varDelta}  {{\mathit\Delta}}

\newcommand{\Deltait}{\varDelta}

\newcommand{\setS}{\mathcal{S}}

\newcommand{\subM}{_{\scriptscriptstyle M}}

\newcommand{\Ascr}{{\mathcal A}}

\newcommand{\Fscr}{{\mathcal F}}

\newcommand{\Mscr}{{\mathcal M}}

\newcommand{\Sscr}{{\mathcal S}}

\newcommand{\Wscr}{{\mathcal W}}

\newcommand{\Hbar}{\skew5\bar H}

\newcommand{\Htilde}{\widetilde H}

\newcommand{\uhat}{\skew3\widehat u}

\newcommand{\what}{\skew3\widehat w}

\makeatother

\begin{document}
\title{Direction of negative curvature for regularized SQP}
\author{Phillip Gill, Vyacheslav Kungurtsev, Daniel Robinson}
\maketitle
\section{Introduction}
This note discusses the computation and use of a direction of negative curvature in 
  the regularized sequential quadratic programming 
 primal-dual augmented Lagrangian method (pdSQP) of Gill and Robinson 
\cite{GilR10}, \cite{GilR11} for the 
purpose of ensuring convergence towards second-order optimal points.
Section \ref{s:comp} discusses how to compute a direction of negative curvature 
using appropriate matrix factorizations. 
Section \ref{s:changes} discusses the specific relevant changes to the 
algorithm. Section \ref{s:conv1} discusses the changes in the convergence 
results established by Gill and Robinson \cite{GilR11}, showing that the desired 
convergence results continue to hold. 
 Section 
\ref{s:conv2} discusses global convergence to points satisfying the 
second-order necessary optimality conditions. 
\section{Direction of negative curvature}\label{s:comp}
\subsection{The active-set estimate}
An index set $\Wscr_k$ is maintained that consists of the variable indices that 
estimate which components of $x$ on their bounds. This set determines 
the the space in which
to calculate the directions of negative curvature. 
The tolerance for an index to be in $\Wscr_k$ must converge to zero. 
A test such as $i\in\Wscr_k$ if $[x_k]_i\le \min\{\mu_k,
\epsilon_a\}$, would be appropriate for the purpose of forming a $\Wscr_k$ for 
convexification, initializing the QP, and obtaining a direction of negative 
curvature. Otherwise, it would be necessary to use  three different 
factorizations.

\subsection{Calculating the direction}
Recall that in pdSQP, the QP must use a Lagrangian Hessian $\Htilde$ such that 
$\Htilde+\frac{1}{\mu}J\T J$ is positive definite.
The process for forming the requisite $\Htilde$, as well as calculating 
a direction of negative curvature
begins with the inertia-controlling factorization of the KKT matrix 
(see Forsgren \cite{For02}). 
Consider the KKT matrix, 
\begin{equation}\label{eqn:kkt}
\pmat{
H_F & J^T_F \\
J_F & -\mu I_{|F|}
},
\end{equation}
with $F$ the set of estimated free variables (those not in $\Wscr_k$), 
and $I_{|F|}$ the identity matrix with $|F|$ rows and columns.

The algorithm begins an LBL$^T$ factorization of the 
KKT matrix, where $L$ is lower triangular and $B$ is a symmetric diagonal 
with $1\times 1$ and $2\times 2$ diagonal blocks. Standard pivoting strategies 
are described in the literature (see Bunch and Parlett \cite{BunP71}, Fletcher 
\cite{Fle76}, and Bunch and Kaufman \cite{BunK80}). Let the lower-right block 
be defined as $D=-\mu I_{|F|}$. 

At step $k$ of the factorization, let the partially factorized matrix have the 
following structure:
$$
\pmat{L_1 & 0 \\ L_2 & I} \pmat{B & 0 \\ 0 & A} \pmat{L^T_1 & L^T_2 \\ 0 & I},
$$
with $L_1$ being lower triangular, $I$ the identity of appropriate size, and $A$ the 
matrix remaining to be factorized. Let $A$ be partitioned as $A=\pmat{a & b\T \\ b & C}$. 
If the top left element is chosen as a $1\times 1$ pivot, at the next step, 
$$
\pmat{L_1 & 0 & 0 \\ L_3 & 1 & 0 \\ L_4 & a^{-1}b & I} \pmat{B & 0 & 0 \\ 0 & a & 0 \\ 
0 & 0 & C-ba^{-1} b\T} \pmat{L^T_1 & L^T_3 & L^T_4 \\ 0 & 1 & a^{-1} b\T \\ 0 & 0 & I}.   
$$

Let $S=C-ba^{-1} b$ be the \emph{Schur complement} of the factorization. 
The matrix $S$ is factorized at the next step.

For inertia control, this factorization has two stages. In the first 
stage, we restrict 
the factorization to allow only for pivots of type $H^+$, $D^-$ or $HD$.
This means that an element $(i,j)$ of $H$
is selected such that $H_{ij}>0$, a diagonal element of $D$ is selected, 
or $(i_1,i_2,j_1,j_2)$ is selected such that $(i_1, j_1)$ is an element of 
$H$, $(i_2,j_2)$ is an element of $D$ and $S_k[(i_1,i_2),(j_1,j_2)]$ has 
mixed eigenvalues. This procedure is continued until there are no such 
remaining pivots.

The KKT matrix can be partitioned as
$$
\pmat{ H_{11} & H_{12} & J^T_1 \\
H_{21} & H_{22} & J^T_2 \\
J_1 & J_2 & -\mu I},
$$
where, all of the pivots have come from the rows and 
columns of $H_{11}$, $J_1$, and $-\mu I$. 
At the end of the first stage, the factorization can be written as:
\begin{equation}\label{eqn:factor}
\pmat{L_1 & 0 \\ L_2 & I} \pmat{B & 0 \\ 0 & H_{22}-K_{21}K^{-1}_{11} K_{12} } 
\pmat{L^T_1 & L^2_2 \\ 0 & I}. 
\end{equation}
Let $S=H_{22}-K_{21}K^{-1}_{11} K_{12}$. 
Proposition 3 of Forsgren \cite{For02} shows that if $\delta I$ is added 
to $H_{22}$ such that $\delta>||S||$ 
then $K_F$ has the correct inertia.
In practice this $\delta$ is excessively large for the purpose of constructing 
the appropriate matrix with the required eigenvalues, 
but this result does indicate that such a constant exists. 

Instead of proceeding to the second phase of this factorization,
the procedure of Lemma 2.4 in Forsgren et al. \cite{ForGM93} is applied to 
$S$ to compute $\uhat$, 
a direction of negative curvature for $S$.
The procedure to calculate this $\uhat$ is as follows: 

Let $\rho=\max_{i,j}|S_{ij}|$ with $|S_{qr}|=\rho$. Define $\uhat$ as the solution to:
\begin{equation}\label{eqn:uhat}
\pmat{
L_1 & 0 \\
L_2 & I} \uhat = \sqrt{\rho}h,
\end{equation}
where 
\begin{equation*}
h 
= \left\{ 
\begin{array}{l@{\hspace{3ex}}l}
e_q & \mbox{if } q=r, \\
\frac{1}{\sqrt{2}}(e_q-\sgn(b_{qr})e_r) & \mbox{otherwise.}
\end{array}
\right.
\end{equation*}

This $\uhat$ satisfies $\uhat\T S \uhat\le \gamma 
\lambda\submin(S) ||\uhat||^2$, with $\gamma$ independent of $S$.

The following bounds are important for the subsequent second-order convergence theory.

\begin{lemma}\label{lem:curvbounds}
Let $\uhat$ be defined as in \eqref{eqn:uhat}, $S$ be the Schur complement 
of the partially factorized matrix \eqref{eqn:factor}, $J_F$ and $H_F$ defined 
as in \eqref{eqn:kkt}, and $Z$ a matrix consisting of columns for the basis 
of the null-space of $J_F$, then

$$\frac{\hat{u}\T S \hat{u}}{\gamma ||\hat{u}||^2}\le \lambda\submin(S) \le 
\lambda\submin(H_F+\frac{1}{\mu}J^T_F J_F ) \le \lambda\submin(Z\T H_F Z).
$$
\end{lemma}
\begin{proof}
Lemma 2.4 in Forsgren et al. \cite{ForGM93} directly implies that
$\hat{u}\T S \hat{u}/\gamma ||\hat{u}||^2\le \lambda\submin(S)$.

The proof that $\lambda\submin(S) \le 
\lambda\submin(H_F+\frac{1}{\mu}J^T_F J_F )$  is given 
in the proof of Theorem 4.5 in Forsgren and Gill \cite{ForG98}. For the final inequality, 
let $w=Zv$, with $Z\T H_F Zv = \lambda\submin(Z\T H_F Z)v$ and $||v||=1$. 
Then 
$$
\lambda\submin(H_F+\frac{1}{\mu}J^T_F J_F) \le \frac{w\T (H_F+\frac{1}{\mu}J^T_F J_F) w}{w\T w}
=w\T H_F w=v\T Z\T 
H_F Z v=\lambda\submin(Z\T H_F Z).
$$
\end{proof}

\section{Implementing Directions of Negative Curvature}\label{s:changes}
\subsection{Step of negative curvature}\label{s:step}
Several changes must be made to the algorithm of Gill and 
Robinson \cite{GilR11}. In order 
to minimize the number of factorizations, the computation 
of the direction of negative curvature should be followed by a test 
of second-order optimality. In addition, it is necessary that 
the direction of negative curvature is bounded, and a feasible direction 
with respect to both the linearized equalities and the bound constraints. 
Finally, the line-search must be extended to allow for this additional step 
of negative curvature.

In the description below, the subscript $k$ denoting the step number in the 
sequence of iterations is suppressed.

The following procedure satisfies these requirements.
\begin{enumerate}
\item The first step computes the direction of negative curvature for the free KKT-matrix
as described in Section \ref{s:comp}, denoted as $\uhat_F$, then defines $\uhat$ to be 
$[\uhat]_F=\uhat_F$ and 
$[\uhat]_A=0$. If no such direction of negative curvature exists, then $\uhat$ is 
set to zero.

\item The second step uses $\uhat$ in a test of second-order optimality. 
This is described in Section \ref{s:opt}.

\item The corresponding change in the multipliers corresponding to 
the definition for $\uhat$ is defined as
$\what=-\frac{1}{\mu}J\uhat$. 
This ensures that the linearized equality constraints are satisfied, i.e., 
$$
0=Jp+c+\mu q=J(p+\uhat)+c+\mu (q-\frac{1}{\mu}J\uhat).
$$
The final resulting $(u,w)$ is shown below in Section \ref{s:merit} to be a 
direction of negative curvature for 
$\nabla^2 \Mscr$.

\item  Since both $(\uhat,\what)$ and $-(\uhat,\what)$ are directions of negative 
curvature, the sign is chosen so that the step is a descent direction for $\nabla \Mscr$, 
i.e.
$\nabla \Mscr\T \pmat{\uhat \\ \what}\le 0$.

\item Compute $\Deltait v=(p,q)$, the solution of the convex QP.
\item The direction of negative curvature is 
scaled so that 
it is both bounded by $\max(u\submax,2||p||)$ and also, in conjunction with the QP step, 
satisfies the bound constraints $x\ge 0$.

Specifically, $u$ and $w$ are set as $u=\beta \uhat$ and $w=\beta \what$,
where $$
\beta=\left\{\max \hat{\beta} \mid
x+p+\hat{\beta} \uhat\ge 0, \,||\hat{\beta}\uhat||\le\max(u\submax,2||p||)\right\}.
$$

Note that this implies that if $[x+p]_i=0$ and $[u]_i<0$, 
then $u$ is set to zero. 

\end{enumerate}

\subsection{Optimality measures}\label{s:opt}
Recall that in Gill and Robinson \cite{GilR11}, with 
\begin{equation*}
  \phi_S(v) = \eta(x) + 10^{-5}\omega(v)
  \wordss{and}
  \phi_L(v) = 10^{-5}\eta(x) + \omega(v),
\end{equation*}
where
\begin{equation*} 
  \eta(x)     = \norm{c(x)}                    \wordss{and}
  \omega(x,y) = \left\| \min\big( x, g(x) - J(x)^T y \big) \right\|,
\end{equation*}
an iterate is an S-iterate if $\phi_S(v)\le \half\phi^{max}_S$ and an 
L-iterate if $\phi_L(v)\le\half\phi^{max}_L$. Otherwise, 
an iterate is an M-iterate if 
$$
||\nabla_y \Mscr(v_{k+1};y^E_k,\mu^R_k)||\le \tau_k \text{ and } 
||\min(x_{k+1},\nabla_x \Mscr^\nu(v_{k+1};y^E_k,\mu^R_k))||\le \tau_k.
$$
If none of these conditions hold, then an iterate $v_k$ is an F-iterate.

In order to force convergence to a second-order optimal point, it is necessary 
to change the function $\omega(x,y)$ that appears in $\phi_S$ and $\phi_L$, as well 
as the test for an iteration being an M-iterate.

Ideally, the minimum eigenvalue of $H$ in the null-space for $J_F$ should be found, 
as well as the minimum eigenvalue of $\nabla^2_{xx} \Mscr$. However, this would 
require extensive computation. Instead, these quantities are estimated based on 
the value of the negative curvature. Recall that 
$$
\frac{\uhat\T(H+\frac{1}{\mu}J\T J) \uhat}{\gamma ||\uhat||^2}\le \lambda\submin(H+\frac{1}
{\mu}J\T J),$$
where we suppress the suffix $F$. Since $\gamma$ is bounded from below and above, 
if $ \uhat\T(H+\frac{1}{\mu}J\T J) \uhat/||\uhat||^2\to 0$, the 
estimate for $\uhat$ implies 
$\lim\lambda\submin(H+\frac{1}
{\mu}J\T J) \ge 0.$ 
Hence, the test for M-iterate optimality is changed to:
\begin{eqnarray*}
& ||\nabla_y \Mscr(v_{k+1};y^E_k,\mu^R_k)||\le \tau_k  \\
\text{and } & ||\min(x_{k+1},\nabla_x \Mscr^\nu(v_{k+1};y^E_k,\mu^R_k))||\le \tau_k \\
\text{and } &  \frac{\uhat_{k+1}\T(H+\frac{1}{\mu}J\T J) \uhat_{k+1}}{||\uhat_{k+1}||^2} 
\ge \tau_k.
\end{eqnarray*}
Similarly, for the filter functions, 
$$
\phi_S(v)=\eta(x)+10^{-5}\omega(v) \text{ and } \phi_L(v)=10^{-5}\eta(x)+\omega(v)
$$
the optimality tests become
$$
\eta(x) = ||c(x)|| \text{ and } \omega(x,y)=\min(||\min(x,g(x)-J(x)\T y),
-\frac{\uhat_{k+1}\T(H+\frac{1}{\mu}J\T J) \uhat_{k+1}}{||\uhat_{k+1}||^2}).
$$

\subsection{Merit function}\label{s:merit}
The line-search must also be changed to include the direction of negative curvature. 
First, it will be shown that the full primal-dual step is a step of negative 
curvature for the merit function Hessian.
\begin{lemma} The vector $(u,w)$ defined as in \ref{s:step} is a direction of negative curvature for 
$\nabla^2 \Mscr$.
\end{lemma}
\begin{proof}
Consider the calculation of $\pmat{u \\ w}\T \nabla^2\Mscr \pmat{u \\ w}$.
\begin{equation*}
\begin{array}{r@{\hspace{3pt}}l}
& \pmat{u \\ w}\T \pmat{H+\frac{1}{\mu}(1+\nu)J\T J & \nu J\T \\ \nu J & \nu \mu I}
\pmat{u \\ w} \\ [1ex]
= & \pmat{u \\ w}\T \pmat{Hu+\frac{1}{\mu}(1+\nu)J\T J u +\nu J\T w \\ \nu J u +\nu
\mu w}  \\ [1ex]
= &u\T H u+\frac{1}{\mu}(1+\nu)u\T J\T J u+2\nu u\T J\T w+\nu\mu||w||^2.
\end{array}
\end{equation*}
From the definition above, $u=\beta\uhat$ and $\uhat\T (H+\frac{1}{\mu}J\T J)\uhat
\le \gamma\lambda\submin(H+\frac{1}{\mu}J\T J)||\uhat||^2$, so multiplying both sides 
by $\beta^2$, the expression becomes 
 $u\T (H+\frac{1}{\mu}J\T J)u\le \gamma\lambda\submin(H+\frac{1}{\mu}J\T J)
||u||^2$. Let $\bar{\gamma}=\gamma\lambda\submin(H+\frac{1}{\mu}J\T J)$.

Using $w\triangleq-\frac{1}{\mu}Ju$,
\begin{equation*}
\begin{array}{r@{\hspace{3pt}}l}
& u\T H u+\frac{1}{\mu}(1+\nu)u\T J\T J u+2\nu u\T J\T w+\nu\mu||w||^2 \\
\le & -\bar{\gamma} ||u||^2-2\frac{\nu}{\mu}u\T J\T J u+\frac{\nu}{\mu}||Ju||^2 \\
=& -\bar{\gamma} ||u||^2-\frac{\nu}{\mu}||Ju||^2  \\
\le & -\bar{\gamma} ||u||^2 -\nu\mu||w||^2.
\end{array}
\end{equation*}
\end{proof}

For the line-search, let $R_k \triangleq u^T_k \nabla^2 \Mscr^\nu (v_k;y^E_k,\mu^R_k) u_k\le 0$. 
Define $\alpha_k=2^{-j}$ such that
\begin{equation}\label{eq:merit1}
\Mscr^\nu(v_k+\alpha_ku_k+\alpha^2_k\Deltait v_k ;y^E_k,\mu^F_k) \le \Mscr^\nu+
\alpha_k^2 \eta_S N_k+\alpha_k \eta_S R_k.
\end{equation}

Letting $\alphabar \triangleq \min(\alpha\submin,\alpha_k)$ and 
$\muhat \triangleq \max\big( \half\mu_k,\mu^{R}_{k+1}\big)$, the update 
for the penalty parameter becomes:
\begin{equation} \label{mu-update-2}
  \mu_{k+1} =
     \left\{\begin{array}{ll}
    \mu_k, &
          \Mscr^\nu(v_{k+1};y^E_k, \mu_k)
          \le \Mscr^\nu(v_k ;y^E_k, \mu_k)
         + \alphabar \eta_S R_k  
 + \alphabar^2 \eta_S N_k \\[1ex]
    \muhat, & \hbox{otherwise,}
            \end{array}
     \right.
\end{equation}



\section{Consistency with established convergence theory}\label{s:conv1}

In their first-order analysis, Gill and Robinson \cite{GilR11} make the following assumptions:
\begin{assumption} \label{ass-1} Each $\Hbar(x_k,y_k)$ is chosen so that
  the sequence $\{\Hbar(x_k,y_k)\}_{k\ge 0}$ is bounded, with
  $\{\Hbar(x_k,y_k) + (1/\mu^R_k)J(x_k)\T J(x_k)\}_{k\ge 0}$ uniformly
  positive definite.
\end{assumption}

\begin{assumption} \label{ass-2}
  The functions $f$ and $c$ are twice continuously differentiable.
\end{assumption}

\begin{assumption}  \label{ass-3}
  The sequence $\{x_k\}_{k\geq 0}$ is contained in a compact set.
\end{assumption}

Since $\nabla \Mscr^\nu$ does not involve any term involving the objective 
or constraint Hessians, much of the first-order convergence theory 
holds. Incorporating the direction of negative curvature, 
Theorem 4.1 changes to:
\begin{theorem}  If there exists an integer $\khat$ such that $\mu^R_k\equiv 
\mu^R >0$ and $k$ is an $\Fscr$-iterate for all $k\ge \khat$, then the following 
hold:
\begin{enumerate}
\item $\left\{||\Deltait v_k||+||u_k||\right\}_{k\ge \khat}$ is bounded away from zero
\item There exists an $\epsilon$ such that for all $k\ge \khat$, it 
holds that
\begin{eqnarray*}
\nabla \Mscr^\nu (v_k;y^E_k,\mu^R_k)\T \Deltait v_k)\le -\epsilon  \mbox{ or }  
u_k\T \nabla^2\Mscr^\nu(v_k;y^E_k,\mu^R_k)u_k\le -\epsilon.
\end{eqnarray*}
\end{enumerate}
\end{theorem}
\begin{proof}
If all iterates $k\ge \khat$ are $\Fscr$-iterates, then,
$$
\tau_k \equiv \tau > 0, \mgap \mu^R_k=\mu^R,\text{ and } y^E_k=y^E \text{ for all} 
k\ge \khat
$$

Proof of the first result: Assume the contrary, i.e., there exists a subsequence 
 $\Sscr_1 \subset \left\{k\mid k\geq \khat\right\}$ 
such that $\lim_{k\in\Sscr_1} \Deltait v_k=0$ and $\lim_{k\in\Sscr_1} u_k=0$. 
The solution $\Deltait v_k$ to the QP subproblem satisfies 
$$
\pmat{z_k \\ 0} = H^\nu_M(v_K;\mu^R)\Deltait v_k+\nabla \Mscr^\nu(v_k;y^E,\mu^R) 
\text{ and } 0=\min(x_k+p_k,z_k).
$$
As $H^\nu_M$ is uniformly bounded, eventually for some $k\in \Sscr_1$ 
sufficiently large, $\Deltait v_k$ satisfies the first-order conditions of 
an M-iterate, i.e.,
$$
||\nabla_y \Mscr(v_{k+1};y^E_k,\mu^R_k)||\le \tau_k \text{ and } 
||\min(x_{k+1},\nabla_x \Mscr^\nu(v_{k+1};y^E_k,\mu^R_k))||\le \tau_k.
$$

In the construction of $u_k$, recall that $||u||$ is the largest possible 
value, subject to an upper bound, that is feasible. This implies that if 
$\lim u_k\to 0$, then eventually, $u$ is constrained by feasibility, or set to 
zero.

In the first case, i.e. the limiting upper bound constraint on $u_k$
 must be $x_k+p_k+u_k\ge 0$, eventually since $u_k\to 0$ and $p_k \to 0$, 
if $i$ is a blocking bound for $u_k$, $x_i\le \min(\mu,\epsilon_x)$ and
$i\in \Wscr_k$, which implies that 
$[u]_i\equiv 0$. Hence, by 
construction and the fact that the set of possible indices is finite, 
$u_k$ is eventually identically zero. This implies that the second-order 
conditions of an M-iterate are also satisfied trivially, i.e.,
$$
\frac{\uhat_{k+1}\T(H+\frac{1}{\mu}J\T J) \uhat_{k+1}}{||\uhat_{k+1}||^2} 
\ge \tau_k,
$$
and $\mu^R_k$ is decreased.
This contradicts the assumption that $\mu^R_k$ is held fixed at $\mu^R_k\equiv \mu^R$ 
for all $k\ge \khat$.


Proof of part (2):
Assume, to the contrary, that there
exists a subsequence $\setS_2$ of $\{k:k\geq\khat\}$ such that
\begin{equation} \label{inner-to-zero}
    \lim_{k\in \setS_2} \Grad\Mscr^\nu(v_k ; y^E,\mu^R)\T \Deltait v_k = 0
\end{equation}
and 

$$
 \lim_{k\in \setS_2} u_k\T \nabla^2\Mscr^\nu(v_k;y^E_k,\mu^R_k)u_k = 0.
$$

Consider the matrix
$$
  L_k = \pmat{ I & 0 \\ \frac{1}{\mu^R}J_k & I}.
$$
Since the $\Deltait v= 0$ is feasible and $\Deltait v_k$ a solution for the convex
problem, 
it follows that
\begin{eqnarray*}
   -\Grad\Mscr^\nu(v_k ; y^E,\mu^R)\T \Deltait v_k
    &\ge& \half \Deltait v_k^T H\subM^\nu(v_k ;\mu^R) \Deltait v_k    \\[1ex]
    & = & \half \Deltait v_k^T L_k^{-T}L_k^T H\subM^\nu(v_k; \mu^R) L_k L_k^{-1}\Deltait v_k\\[1ex]
    & = & \pmat{ p_k             \\[2pt]
                 q_k+\frac{1}{\mu^R  }
           J_k p_k}^T
           \pmat{\Hbar_k + \frac{1}{\mu^R}J_k^T J_k &  0     \\[2pt]
                    0                                    & \nu\mu^R  }
           \pmat{ p_k                      \\[2pt]
                  q_k+\frac{1}{\mu^R} J_k p_k     }   \\[1ex]
\end{eqnarray*}

Since $H^\nu_M$ is bounded, 
$$
 \Deltait v_k^T L_k^{-T}L_k^T H\subM^\nu(v_k; \mu^R) L_k L_k^{-1}\Deltait v_k
   \ge \bar{\lambda}\submin\norm{p_k}^2 + \nu\mu^R\norm{q_k+(1/\mu^R) J_k p_k}^2,
$$

for some $\bar{\lambda}\submin > 0$.  Combining this
with~\eqref{inner-to-zero} it follows that
$$
  \lim_{k\in\setS_2} p_k
   = \lim_{k\in\setS_2} \Big(q_k+ \frac{1}{\mu^R} J_k p_k\Big) = 0,
$$
in which case $\lim_{k\in\setS_2} q_k = 0$. Hence $\Deltait v_{k\in\setS_2}\to 0$.

Since $\lim_{k\in\Sscr_2}u\T_k \nabla^2
\Mscr^\nu(x_k,y_k;y^E,\mu) u_k =0$, there exists a $\khat_2$, such that
for all $k\ge \khat_2$, $u^T_k \nabla^2\Mscr^\nu(x_k,y_k;y^E,\mu) u_k /\gamma||u_k||^2> -\tau$
or $u_k\to 0$. The former, by the same argument as for 
part (1), together with $\Deltait v_k\to 0$, implies that eventually 
$k$ is an M-iterate. The latter, together with $\lim 
\Deltait_k=0$, contradicts the statement of part (1) of the theorem, 
so part (3) must hold.
\end{proof}

The proofs of the first result of Theorem 4.1 and Theorem 4.2 of 
Gill and Robinson \cite{GilR11} do not change.

\section{Global convergence to second-order optimal points}\label{s:conv2}
\subsection{Filter Convergence}
\begin{definition} The Weak Constant Rank (WCR) condition holds at $x$ if 
there is a neighborhood $M(x)$ for which the rank of $\pmat{J(z)\\E_{\Ascr}\T}$ is 
constant for all $z\in M(x)$, where $E_\Ascr$ is the columns of the identity 
corresponding to the indices of $x$ active at $x$ (as in $i\in\Ascr$ if $x_i=0$).
\end{definition}

\begin{theorem}
Assume there is a subsequence $v_k$ of S- and L-iterates converging  
to $\vstar$, with $\vstar=(\xstar,\ystar)$ satisfying the first-order KKT conditions. 
Furthermore, assume that MFCQ and WCR hold at $\vstar$. Then $\vstar$ satisfies 
the necessary second-order necessary optimality conditions.
\end{theorem}
\begin{proof}
Let $d\in T(\xstar)\equiv \left\{d\mid J(\xstar)d=0 \mbox{ and } E_{\wstar}\T d=
0\right\}$ with $||d||=1$. By Lemma 3.1 of Andreani et al. (\cite{AMS10b}) 
there exists $\{d_k\}$ such that $d_k\in T(x_k)$ and $d_k \to d$, where 
$$
T(x_k) = \left\{d\mid J(x_k)d=0 \mbox{ and } E_{w^*}\T d=
0\right\}.
$$
Without loss of generality, we may 
let $||d_k||=1$. Since $x_k\to x^*$, eventually 
$\Wscr_k=\Ascr^*$, where $\Ascr^*$ is the active set at $\xstar$.
Then, by the definition of the S- and L-iterates, and Lemma \ref{lem:curvbounds}, $d_k\T (\nabla^2 
f(x_k)+\sum y_k \nabla^2 c(x_k)) d_k > \lambda\submin(Z_k H_kZ_k)>-\xi_k$, where $0<\xi_k\to 0$. 
Taking limits, it follows that $d\T 
(\nabla^2 f(x_k)+\sum \ystar \nabla^2 c(\xstar))d\geq 0$.
\end{proof}


\bibliographystyle{plain}
\bibliography{../refs,references}
\end{document}